\documentclass{amsart}
\usepackage{amsfonts}

\newtheorem{theorem}{Theorem}
\theoremstyle{plain}

\newtheorem{conjecture}{Conjecture}

\numberwithin{equation}{section}
\input{tcilatex}

\begin{document}
\title[ON THE SOLUTIONS OF THE EQUATION]{ON THE SOLUTIONS OF THE EQUATION $%
\left( 4^{n}\right) ^{x}+p^{y}=z^{2}$}
\author{Bilge Peker}
\address{Elementary Mathematics Education Programme, Ahmet Kelesoglu
Education Faculty, Konya University, KONYA, TURKEY.}
\email{bilge.peker@yahoo.com}
\author{Selin (INAG) CENBERCI}
\address{Secondary Mathematics Education Programme, Ahmet Kelesoglu
Education Faculty, Konya University,KONYA, TURKEY.}
\email{inag\_s@hotmail.com}
\date{January 21, 2012}
\subjclass[2000]{Primary 11D61}
\keywords{Exponential Diophantine Equation}

\begin{abstract}
In this paper, we gave solutions of the Diophantine equations $%
16^{x}+p^{y}=z^{2},$ $64^{x}+p^{y}=z^{2}$ where $p$ is an odd prime, $n\in 
\mathbb{Z}
^{+}$and $x,y,z$ are non-negative integers. Finally we gave a generalization
of the Diophantine equation $\left( 4^{n}\right) ^{x}+p^{y}=z^{2}.$
\end{abstract}

\maketitle

\textbf{1. Introduction}

There are lots of studies about the Diophantine equation of type $%
a^{x}+b^{y}=c^{z}.$ In 1999, Z. Cao $\left[ 3\right] $ proved \ that this
equation has at most one solution with $z>1.$ In 2005, D. Acu $\left[ 2%
\right] $ showed that the Diophantine equation $2^{x}+5^{y}=z^{2}$ has
exactly two solutions in non-negative integers, i.e. $\left( x,y,z\right)
\in \left\{ \left( 3,0,3\right) ,\left( 2,1,3\right) \right\} .$ J. Sandor $%
\left[ 8\right] $ studied on the Diophantine equation $4^{x}+18^{y}=22^{z}.$
In 2011, A. Suvarnamani $\left[ 11\right] $ considered the Diophantine
equation $2^{x}+p^{y}=z^{2}$ where \textit{p }is a prime and $x,y,z$ are
non-negative integers. A. Suvarnamani, A. Singta and S. Chotchaisthit $\left[
12\right] $ found solutions of the Diophantine equations $4^{x}+7^{y}=z^{2}$
and $4^{x}+11^{y}=z^{2}.$ In 1657, Frenicle de Bessy $\left[ 7\right] $
solved a problem possed by Fermat: if $p$ is an odd prime and $n\geq 2$ is
integer, then the equation $x^{2}-1=p^{n}$ has no integer solution. S.
Chotchaisthit $\left[ 4\right] $ showed all non-negative integer solutions
of $4^{x}+p^{y}=z^{2}$ where $p$ is a prime number. B. Peker, S. I. Cenberci 
$\left[ 6\right] $ studied on the solutions of the Diophantine eqations of
some types $\left( 2^{n}\right) ^{x}+p^{y}=z^{2}$ where $p$ is an odd prime
and $x,y,z$ are non-negative integers$.$

In this study, we gave solutions of the Diophantine equations $%
16^{x}+p^{y}=z^{2},$ $64^{x}+p^{y}=z^{2}$ where $p$ is an odd prime and $%
x,y,z$ are non-negative integers. Then we gave a generalization of the
Diophantine equation $\left( 4^{n}\right) ^{x}+p^{y}=z^{2}.$

\textbf{2. Preliminaries}

In this study, we use Catalan's Conjecture $\left[ 5\right] $ and a theorem.
Now we give them.

\begin{conjecture}
(Catalan) The only solution in integers $a>1,\ b>1,\ x>1$ and $y>1$ of the
equation $a^{x}-b^{y}=1$ is $a=y=3$ and $b=x=2.$
\end{conjecture}

\begin{theorem}
$\left[ 7\right] $ \ If $p$ is an odd prime and $n\geq 2$ is integer, then
the equation $x^{2}-1=p^{n}$ has no integer solution.
\end{theorem}

Now we give our theorems.

\begin{theorem}
When $p$ is an odd prime and $k,x,y,z$ are non-negative integers, solutions
of the Diophantine equation%
\begin{equation*}
16^{x}+p^{y}=z^{2}\ \ \ \ \ \ \ \ \ \ \ \left( 1\right)
\end{equation*}

are given in the following two cases:
\end{theorem}

\ \ \ \ \ \textit{Case 1. For} $p=3,$ $\left( x,y,z\right) =\left(
1,2,5\right) ,$

\ \ \ \ \ \textit{Case 2. For} $p=1+2^{2k+1},$ $\left( x,y,z\right) =\left(
k,1,2^{2k}+1\right) .$

\ \ \ \ \ \textit{\ }\ \ \ \ \ \ \ \ \ \ \ \ \ \ \ \ \ \ \ \ \ \ \ \ \ \ \ \
\ \ \ \ \ \ \ \ \ \ \ 

\begin{proof}
If we consider $y>0,$ then we get%
\begin{equation*}
\ z^{2}-4^{2x}=p^{y}
\end{equation*}

i.e.$\ $%
\begin{equation*}
(z-4^{x})(z+4^{x})=p^{y}
\end{equation*}

where $z-4^{x}=p^{v}$ and $z+4^{x}=p^{y-v},$ $y>2v$ and $v$ is a
non-negative integer. Then we get $p^{y-v}-p^{v}=2^{2x+1}$ or $%
p^{v}(p^{y-2v}-1)=2^{2x+1}.$

If $v=0$, we obtain $p^{y}-1=2^{2x+1}$ or $p^{y}-2^{2x+1}=1$. From the $%
Catalan^{\prime }s$ $Conjecture$, it is obvious that $p=3,\ y=2\ $and $%
2x+1=3,\ x=1.$ Hence a solution of the equation $\left( 1\right) $ is $%
(x,y,z,p)=(1,2,5,3).$

For $y=1$, we obtain $p-2^{2x+1}=1,$ i.e. $p=1+2^{2x+1}$. Hence solutions of
the equation $\left( 1\right) $ are $(x,y,z)=(x,1,2^{2x}+1).$

For $y=0$, we get $z^{2}-4^{2x}=1,$ i.e. $z^{2}-2^{4x}=1$ which is no
solution when $z=0$ or $z=1.$ From the $Catalan^{\prime }s$ $Conjecture$, $%
z=3$ and $4x=3$ which is impossible.

If we consider $x=0,$ then we have%
\begin{equation*}
p^{y}+1=z^{2}
\end{equation*}

i.e.%
\begin{equation*}
z^{2}-p^{y}=1.
\end{equation*}

From the $Catalan^{\prime }s$ $Conjecture$, it is obvious that $p=2,\ y=3\ $%
and $z=3.$ Hence a solution of the equation $\left( 1\right) $ is $%
(x,y,z,p)=(0,3,3,2).$ But this is not a solution, since we assume that $p$
is an odd prime.

Consequently, it is easy to see that $\left( x,y,z,p\right) =\left(
1,2,5,3\right) $ or $(x,y,z,p)=(k,1,2^{2k}+1,1+2^{2k+1})$ where $k$ is a
non-negative integer are solutions of the equation $16^{x}+p^{y}=z^{2}.$
\end{proof}

\begin{theorem}
When $p$ is an odd prime and $k,x,y,z$ are non-negative integers, solutions
of the Diophantine equation%
\begin{equation*}
64^{x}+p^{y}=z^{2}\ \ \ \ \ \ \ \ \ \ \ \left( 2\right)
\end{equation*}

is given $\left( x,y,z\right) =\left( k,1,2^{3k}+1\right) $ \ for $%
p=1+2^{3k+1}.$
\end{theorem}

\begin{proof}
By using the proof of the \textit{Theorem 2}, solutions of the equation $%
\left( 2\right) $ are obvious.
\end{proof}

Now, in the following theorem we will give a generalization of the above
Diophantine equations.

\begin{theorem}
When $p$ is an odd prime, $n\in 
\mathbb{Z}
^{+}$ and $k,x,y,z$ are non-negative integers, solutions of the Diophantine
equation%
\begin{equation*}
\left( 4^{n}\right) ^{x}+p^{y}=z^{2}\ \ \ \ \ \ \ \ \ \ \ \left( 3\right)
\end{equation*}

are given in the following two cases:

Case 1. For \ $p=3,$ $\left( x,y,z\right) =\left\{ \left( 2,2,5\right)
,\left( 1,2,5\right) \right\} ,$

\textit{Case 2. For} $p=1+2^{nk+1},$ $\left( x,y,z\right) =\left(
k,1,2^{nk}+1\right) .$
\end{theorem}

\begin{proof}
If we consider $y>0,$ then we get%
\begin{equation*}
\ z^{2}-4^{nx}=p^{y}
\end{equation*}

i.e.$\ $%
\begin{equation*}
(z-2^{nx})(z+2^{nx})=p^{y}
\end{equation*}

where $z-2^{nx}=p^{v}$ and $z+2^{nx}=p^{y-v},$ $y>2v$ and $v$ is a
non-negative integer. Then we get $p^{y-v}-p^{v}=2^{nx+1}$ or $%
p^{v}(p^{y-2v}-1)=2^{nx+1}.$

\ If $v=0$, we obtain $p^{y}-1=2^{nx+1}$ or $p^{y}-2^{nx+1}=1$. From the $%
Catalan^{\prime }s$ $Conjecture$, it is obvious that $p=3,\ y=2\ $and $%
nx+1=3,\ nx=2.$\ This gives us two cases.

First case: If $\ n=1$, then $x=2.$ So we obtain the Diophantine equation $%
4^{x}+p^{y}=z^{2}.$ With this condition, the solution of this equation was
given by Chotchaisthit $\left[ 4\right] $ in the form $\left( x,y,z,p\right)
=$ $\left( 2,2,5,3\right) $.

Second case: If $n=2,$ then $x=1.$ Therefore we obtain the Diophantine
equation $16^{x}+p^{y}=z^{2}$. \ With this condition, the solution of this
equation is given in the $Theorem$\textit{\ }$2$ in the form $\left(
x,y,z,p\right) =\left( 1,2,5,3\right) $.

\ If we consider $y=1$, we obtain $p-2^{nx+1}=1,$ i.e. $p=1+2^{nx+1}$. For $%
n=1,$ solutions of the equation $4^{x}+p^{y}=z^{2}$\ is of the form $\left(
x,y,z,p\right) =$ $\left( r,1,2^{r}+1,2^{r+1}+1\right) $ where r is a
non-negative integer in $\left[ 4\right] $. For $n=2,$ solutions of the
equation $\left( 1\right) $\ is of the form $\left( x,y,z,p\right) =$ $%
\left( k,1,2^{2k}+1,2^{2k+1}+1\right) $ where $k$ is a non-negative integer.
For $n=3,$ solutions of the equation $\left( 2\right) $\ is given by $\left(
x,y,z,p\right) =$ $\left( k,1,2^{3k}+1,2^{3k+1}+1\right) $ where $k$ is a
non-negative integer. If$\ n$ is any positive integer, solutions of the
equation $\left( 3\right) $ is given of the form $%
(x,y,z,p)=(k,1,2^{nk}+1,2^{nk+1}+1)$ where $k$ is a non-negative integer$.$

For $y=0$, we get $z^{2}-4^{nx}=1,$ i.e. $z^{2}-2^{2nx}=1$ which is no
solution when $z=0$ or $z=1.$ For the other cases if we consider the $%
Catalan^{\prime }s$ $Conjecture$, we don't find a solution.

If we consider $x=0,$ we have%
\begin{equation*}
p^{y}+1=z^{2}
\end{equation*}

i.e.%
\begin{equation*}
z^{2}-p^{y}=1.
\end{equation*}

From the $Catalan^{\prime }s$ $Conjecture$, it is obvious that $p=2,\ y=3\ $%
and $z=3.$ Hence a solution of the equation $\left( 3\right) $ is $%
(x,y,z,p)=(0,3,3,2).$ Since $p$ is not an odd prime, this is not a solution
becouse $p$ is an even prime.

Therefore, it is easy to see that $\left( x,y,z,p\right) =\left\{ \left(
1,2,5,3\right) ,\left( 2,2,5,3\right) \right\} $ or $%
(x,y,z,p)=(k,1,2^{nk}+1,1+2^{nk+1})$ where $k$ is a non-negative integer are
solutions of $\left( 4^{n}\right) ^{x}+p^{y}=z^{2}.$
\end{proof}


\bigskip

\bigskip

\begin{thebibliography}{99}
\bibitem{1} D. Acu, On the Diophantine equations of type $a^{x}+b^{y}=c^{z},$
General Mathematics Vol.13, No.1, 67-72, 2005.

\bibitem{2} D. Acu, On a Diophantine equation $2^{x}+5^{y}=z^{2}$, General
Mathematics Vol.15, No.4, 145-148, 2007.

\bibitem{3} Z. Cao, A note on the Diophantine equation $a^{x}+b^{y}=c^{z},$
Acta Arithmetica XCI, No.1, 85-89, 1999.

\bibitem{4} S. Chotchaisthit, On the Diophantine equation $4^{x}+p^{y}=z^{2}$
where $p$ is a prime number, American Jr. of Mathematics and Sciences Vol.1,
No.1, January 2012.

\bibitem{5} P. Mihailescu, Primary cyclotomic units and a proof of Catalan's
Conjecture, J. Reine Angew Math. 572, 167-195, 2004.

\bibitem{6} B. Peker, S. I. Cenberci, On the Diophantine equations of $%
\left( 2^{n}\right) ^{x}+p^{y}=z^{2}$ type, American. Journal of Mathematics
and Sciences Vol.1, No.1, January 2012.

\bibitem{7} P. Ribenboim, Catalan's Conjecture, Academic Press. Inc. Boston,
MA, 1994.

\bibitem{8} J. Sandor, On a Diophantine equation $4^{x}+18^{y}=22^{z}$,
Geometric theorems, Diophantine equations and Arithmetic functions, American
Research Press Rehobot 4, 91-92, 2002.

\bibitem{9} W. Sierpinski, Elementary Theory of Numbers, Warszawa, 1964.

\bibitem{10} J. H. Silverman, A Friendly Introduction to Number Theory, 2nd
ed. Prentice- Hall, Inc. New Jersey, 2001.

\bibitem{11} A. Suvarnamani, Solutions of theDiophantine equation $%
2^{x}+p^{y}=z^{2},$ Int. J. of Mathematical Sciences and Applications Vol.1,
No.3, September 2011.

\bibitem{12} A. Suvarnamani, A. Singta and S. Chotchaisthit, On two
Diophantine equations $4^{x}+7^{y}=z^{2}$ and $4^{x}+11^{y}=z^{2},$ Science
and Technology RMUTT Journal Vol.1, No.1, 2011.
\end{thebibliography}
\end{document}